\documentclass[12pt,leqno]{article}
\usepackage{amssymb,amsmath,amsthm}
\usepackage[all]{xy}
\usepackage[alphabetic,lite]{amsrefs}
\numberwithin{equation}{section}

\usepackage[top=2cm,bottom=2cm,left=2cm,right=4cm,marginparsep=0.3cm,marginparwidth=3cm,includefoot]{geometry}

\usepackage{graphics}

\usepackage[titletoc,title]{appendix}
\usepackage{enumerate}
\usepackage{mathrsfs}
\usepackage{graphics}
\usepackage[usenames,dvipsnames]{xcolor}
\usepackage[colorlinks=true, pdfstartview=FitV, linkcolor=blue,
citecolor=blue, urlcolor=blue]{hyperref}
\usepackage{marginnote}
\usepackage{fancyhdr}
\newenvironment{rouge}
{\relax\color{red}}
{\hspace*{.3ex}\relax}
\newcommand{\ber}{\begin{rouge}}
\newcommand{\er}{\end{rouge}}
\newcommand{\berm}{\begin{rouge}{}\marginnote{\fbox{\scshape\lowercase{M}}}{}}
\newcommand{\berp}{\begin{rouge}{}\marginnote{\fbox{\scshape\lowercase{P}}}{}}

\newenvironment{blue}
{\relax\color{PineGreen}}
{\hspace*{.3ex}\relax}

\newcommand{\beb}{\begin{blue}}
\newcommand{\eb}{\end{blue}}
\newcommand{\bebm}{\begin{blue}{}\marginnote{\fbox{\scshape\lowercase{M}}}{}}
\newcommand{\bebp}{\begin{blue}{}\marginnote{\fbox{\scshape\lowercase{P}}}{}}
\newcommand{\nc}{\newcommand}
\nc{\on}{\operatorname}

\newcommand{\spa}{\vspace{0.3ex}\noindent}

\newlength{\my}
\nc{\noi}{\noindent}

\newtheorem{theorem}{Theorem}[subsection]
\newtheorem{proposition}[theorem]{Proposition}
\newtheorem{lemma}[theorem]{Lemma}
\newtheorem{corollary}[theorem]{Corollary}

\theoremstyle{definition}

\newtheorem{definition}[theorem]{Definition}
\newtheorem{notation}[theorem]{Notation}

\newtheorem{remark}[theorem]{Remark}

\newtheorem{conjecture}[theorem]{Conjecture}

\nc{\Rem}{\begin{remark}}
\nc{\enrem}{\end{remark}}
\nc{\Conj}{\begin{conjecture}}
\nc{\enconj}{\end{conjecture}}
\nc{\Th}{\begin{theorem}}
\nc{\enth}{\end{theorem}}
\nc{\Lemma}{\begin{lemma}}
\nc{\enlemma}{\end{lemma}}
\nc{\Cor}{\begin{corollary}}
\nc{\encor}{\end{corollary}}
\nc{\Def}{\begin{definition}}
\nc{\edf}{\end{definition}}

\nc{\RR}{\mathrm{R}}
\nc{\LL}{\mathrm{L}}

\newcommand{\C}{{\mathbb{C}}}
\newcommand{\N}{{\mathbb{N}}}
\newcommand{\R}{{\mathbb{R}}}

\newcommand{\cor}{{\mathbf{k}}}

\newcommand{\icor}{\mathrm{I}\mspace{1mu}\cor}

\nc{\tA}{\mathcal{A}}
\nc{\rA}{\mathscr{A}}

\def\D{\mathscr{D}}

\def\shd{\mathscr{D}}
\def\she{\mathscr{E}}
\def\shf{\mathscr{F}}

\def\shi{\mathscr{I}}

\def\sho{\mathscr{O}}

\renewcommand{\to}[1][]{\xrightarrow[]{#1}}

\newcommand{\isoto}[1][]{\xrightarrow[#1]%
{{\raisebox{-.6ex}[0ex][-.6ex]{$\mspace{1mu}\sim\mspace{2mu}$}}}}

\newcommand{\To}[1][\rule{1ex}{0pt}]{\xrightarrow{\hs{.6ex}#1\hs{.6ex}}}

\newcommand{\muHom}[1][]{\mathrm{Hom}^\mu_{\raise1.5ex\hbox to.1em{}#1}}
\newcommand{\Hom}[1][]{\mathrm{Hom}_{\raise1.5ex\hbox to.1em{}#1}}
\newcommand{\RHom}[1][]{\RR\mathrm{Hom}_{\raise1.5ex\hbox to.1em{}#1}}
\newcommand{\Ext}[2][]{\mathrm{Ext}_{\raise1.5ex\hbox to.1em{}#1}^{#2}}
\renewcommand{\hom}[1][]{{\mathscr{H}\mspace{-4mu}om}_{\raise1.5ex\hbox to.1em{}#1}}
\newcommand{\rhom}[1][]{{\RR\mathscr{H}\mspace{-3mu}om}_{\raise1.5ex\hbox to.1em{}#1}}
\newcommand{\rhomc}[1][]
{{\mathscr{H}\mspace{-3mu}om}^*_{\raise1.5ex\hbox to.1em{}#1}}

\newcommand{\cihom}[1][]
{{\mathscr{I}\mspace{-3mu}}{hom}^+_{\raise1.5ex\hbox to.1em{}#1}}

\nc{\ihom}[1][]{{\shi\mspace{-3mu}hom}_{\raise1.5ex\hbox to.1em{}#1}}
\nc{\rihom}[1][]{{\mspace{2mu}\mathrm{R}\shi\mspace{-3mu}hom}_{\raise1.5ex\hbox to.1em{}#1}}
\nc{\fihom}[1][]{{\shi\mspace{-3mu}hom}^{\mathrm{E}}_{\raise1.5ex\hbox to.1em{}#1}}
\nc{\FHom}[1][]{{\mathrm{RHom}^{\mathrm{E}}_{\raise1.5ex\hbox to.1em{}#1}}}
\nc{\fhom}[1][]{{\mathscr{H}%
\mspace{-3mu}om}^{\mathrm{E}}_{\raise1.5ex\hbox to.1em{}#1}}
\nc{\Endom}[1][]{{\she\mspace{-3mu}nd}_{\raise1.5ex\hbox to.1em{}#1}}

\nc{\Tam}{{\mathrm{E}}}

\newcommand{\ext}[2][]{{\mathscr{E}xt}_{\raise1.5ex\hbox to.1em{}#1}^{#2}}
\newcommand{\Tor}[2][]{\mathrm{Tor}^{\raise1.5ex\hbox to.1em{}#1}_{#2}}
\newcommand{\tens}[1][]{\mathbin{\otimes_{\raise1.5ex\hbox to-.1em{}{#1}}}}

\newcommand{\ltens}[1][]{\mathbin{\overset{\mathrm{L}}\tens}_{#1}}

\newcommand{\etens}{\mathbin{\boxtimes}}

\newcommand{\Endo}[1][]{\mathrm{End}_{\raise1.5ex\hbox to.1em{}#1}}

\newcommand{\Aut}[1][]{\mathrm{Aut}_{\raise1.5ex\hbox to.1em{}#1}}
\newcommand{\sect}{\Gamma}
\newcommand{\rsect}{\mathrm{R}\Gamma}

\newcommand{\VV}{{\mathsf{V}}}
\newcommand{\VVd}{{\mathsf{V}^*}}

\newcommand{\WW}{{\mathbb{V}}}

\newcommand{\oim}[1]{{#1}_*}
\newcommand{\eim}[1]{{#1}_!}
\newcommand{\roim}[1]{\RR{#1}_*}
\newcommand{\reim}[1]{\RR{#1}_!}
\newcommand{\reiim}[1]{\RR{#1}_{\mspace{1mu}!!}}

\newcommand{\opb}[1]{#1^{-1}}
\newcommand{\pshopb}[1]{#1^{\dag}}

\newcommand{\epb}[1]{#1^{\,!}\,}
\newcommand{\spb}[1]{#1^{*}}

\newcommand{\Dtens}[1][]{\overset{\mathrm{D}}\otimes_{\raise1.5ex\hbox to-.1em{}#1}}
\newcommand{\Detens}[1][]{\overset{\mathrm{D}}\etens_{\raise1.5ex\hbox to-.1em{}#1}}

\nc{\rE}{\mathrm{E}}
\nc{\enh}{\mathsf{E}}

\nc{\EF}[1][]{{}^{\enh}\mspace{-3mu}\shf_{#1}}
\nc{\EFa}[1][]{{}^{\enh}{\mspace{-3mu}\shf^a_{#1}}}
\nc{\FS}[1][]{{}^{\mathrm{S}}{\mspace{-3mu}\shf_{#1}}}
\nc{\FSa}[1][]{{}^{\mathrm{S}}{\mspace{-3mu}\shf^a_{#1}}}
\nc{\Leg}[1][]{{\mathrm{Conv}}{(#1)}}
\nc{\dom}{\mathrm{dom}}
\nc{\domo}{\dom^\circ}

\nc{\hol}{\mathrm{hol}}

\nc{\Mod}{\mathrm{Mod}}
\nc{\rh}{\mathrm{rh}}
\newcommand{\md[1]}{\Mod({#1})}

\newcommand{\mdc[1]}{{\Mod_{\coh}}({#1})}

\nc{\sHH}{\mathscr{H}\mspace{-4mu}\mathscr{H}}

\nc{\sMH}{\mathscr{M}\mspace{-4mu}\mathscr{H}}

\newcommand{\eqdot}{\mathbin{:=}}
\newcommand{\seteq}{\mathbin{:=}}

\newcommand{\cl}{\colon}
\newcommand{\scbul}{{\,\raise.4ex\hbox{$\scriptscriptstyle\bullet$}\,}}

\newcommand{\twX}{{\widetilde{X}}}

\newcommand{\rmH}{{\mathrm{H}}}

\newcommand{\ol}{\overline}

\newcommand{\ba}{\begin{array}}
\newcommand{\ea}{\end{array}}
\newenvironment{nnum}{
\begin{enumerate}
\itemsep=0pt

}
{\end{enumerate}}

\nc{\be}{\begin{enumerate}}
\nc{\ee}{\end{enumerate}}
\newcommand{\bnum}{\begin{enumerate}[{\rm(i)}]}
\newcommand{\enum}{\end{enumerate}}
\newcommand{\banum}{\begin{enumerate}[{\rm(a)}]}
\newcommand{\eanum}{\end{enumerate}}

\newenvironment{myalign}
{\relax\begin{align}}
{\end{align}}
\newenvironment{myalignn}
{\relax\begin{align*}}
{\relax\end{align*}}

\nc{\eq}{\begin{eqnarray}}
\nc{\eneq}{\end{eqnarray}}
\nc{\eqn}{\begin{eqnarray*}}
\nc{\eneqn}{\end{eqnarray*}}

\nc{\eqa}{\begin{myalign}}
\nc{\eneqa}{\end{myalign}}
\nc{\eqan}{\begin{myalignn}}
\nc{\eneqan}{\end{myalignn}}

\newcommand{\set}[2]{\left\{#1 \mathbin{;} #2 \right\}}

\nc{\Proof}{\begin{proof}}
\nc{\QED}{\end{proof}}
\nc{\Prop}{\begin{proposition}}
\nc{\enprop}{\end{proposition}}

\nc{\rop}{{\mathrm{op}}}
\nc{\op}{\rop}
\nc{\tot}{\mathrm{tot}}
\nc{\Op}{{\mathrm{Op}}}
\nc{\Ouv}{{\mathrm{Ouv}}}
\nc{\dist}{{\mathrm{dist}}}
\nc{\LocSyst}{{\mathrm{LocSyst}}}
\nc{\eu}{\mathrm{eu}}
\nc{\hh}{\mathrm{hh}}
\nc{\mueu}{{\mu\eu}}

\DeclareMathOperator{\id}{id}

\newcommand{\Supp}{\on{Supp}}
\newcommand{\Der}[1][]{\mathsf{D}^{#1}}
\newcommand{\Derb}{\Der[\mathrm{b}]}

\newcommand{\Derp}{\Der[+]}

\newcommand{\dT}{{\dot{T}}}

\nc{\wc}[1]{\overset{\mbox{$\scriptscriptstyle\vee$}}{#1}}
\nc{\field}{\cor}
\nc{\bM}{\widehat M}
\nc{\bN}{\widehat N}
\nc{\bX}{{\widehat X}}
\nc{\bS}{\widehat S}
\nc{\bY}{\widehat Y}
\nc{\bL}{\widehat L}
\nc{\bR}{{\ol\R}}
\nc{\bV}{{\ol \VV}}
\nc{\bW}{{\ol \WW}}

\nc{\bVd}{{\ol \VVd}}
\nc{\bWd}{{{\ol \WW}^*}}

\nc{\oM}{{\ol M}}
\nc{\oN}{\ol N}
\nc{\oX}{{\ol X}}
\nc{\oS}{\ol S}
\nc{\oY}{\ol Y}
\nc{\oL}{\ol L}
\nc{\oR}{{\ol\R}}
\nc{\Tl}{\mathrm{L^E}}
\nc{\Tr}{\mathrm{R^E}}

\nc{\sa}{\mathrm{sa}}
\nc{\sas}{\mathrm{sas}}
\nc{\aln}{\mathrm{aln}}
\nc{\zar}{\mathrm{zar}}
\nc{\tz}{\mathrm{tz}}
\nc{\gsa}{\mathrm{gsa}}
\nc{\sat}{\mathrm{sat}}
\nc{\an}{\mathrm{an}}
\nc{\ant}{\mathrm{tan}}
\nc{\temp}{\mathrm{tan}}
\nc{\stan}{\mathrm{sta}}
\nc{\stza}{\mathrm{stz}}

\newcommand{\Msa}{{M_{\sa}}}

\newcommand{\Nsa}{{N_{\sa}}}
\newcommand{\Xsa}{{X_{\sa}}}

\newcommand{\Xcsa}{{X^c_{\sa}}}
\newcommand{\Ycsa}{{Y^c_{\sa}}}
\newcommand{\Xsas}{{X_{\sas}}}
\newcommand{\Xsat}{{X_{\sat}}}

\newcommand{\Ysa}{{Y_{\sa}}}

\newcommand{\Ysat}{{Y_{\sat}}}
\newcommand{\Ysas}{{Y_{\sas}}}
\newcommand{\Xz}{{X_{\zar}}}
\newcommand{\Yz}{{Y_{\zar}}}
\newcommand{\Usa}{{U_{\sa}}}
\newcommand{\Umsa}{{U_{\Msa}}}
\newcommand{\Vnsa}{{V_{\Nsa}}}
\newcommand{\Uxsa}{{U_{\Xsa}}}
\newcommand{\Vysa}{{V_{\Ysa}}}

\newcommand{\rhoz}{\rho_{\zar}}
\newcommand{\rhostz}{\rho_{\stza}}
\newcommand{\rhosa}{\rho_{\sa}}

\newcommand{\rhosas}{\rho_{\sas}}
\newcommand{\rhosat}{\rho_{\sat}}

\newcommand{\rhosasa}{\rho_{\sa\sa}}

\newcommand{\erx}{\eim{\rhosa}}

\newcommand{\Ox}{\sho_{\Xsa}}
\newcommand{\Otx}{\Ot[\Xsa]}
\newcommand{\Otsx}{\Ot[\Xsas]}

\newcommand{\Otsy}{\Ot[\Ysas]}

\newcommand{\Otax}{\Ot[X^\stan]}
\newcommand{\Otempx}{\Ot[X^\temp]}
\newcommand{\Otempy}{\Ot[Y^\temp]}

\newcommand{\Owx}{{\sho^\omega_{\Xsa}}}

\newcommand{\Ocwx}{{\sho^\omega_{\Xcsa}}}
\newcommand{\Ocwy}{{\sho^\omega_{\Ycsa}}}

\newcommand{\Ozx}{{\sho_\Xz}}
\newcommand{\Oty}{\Ot[\Ysa]}
\newcommand{\Ozy}{{\sho_\Yz}}

\newcommand{\Dbv}{{\mathcal D} b^\vee}
\newcommand{\Db}{{\mathcal D} b}
\newcommand{\Dbt}{{\mathcal D} b^{\mspace{2mu}\mathrm t}}
\newcommand{\Dbtv}{{\mathcal D} b^{\mathrm t\vee}}
\newcommand{\Cinft}[1][M]{\mathcal{C}^{\infty,\mathrm t}_{#1}}

\newcommand{\Cinf}[1][M]{\mathcal{C}^{\infty}_{#1}}

\newcommand{\Cinfn}[2]{\mathcal{C}^{\infty,{#2}}_{#1}}

\newcommand{\Ot}[1][X]{\sho^{\mspace{2.5mu}{\mathrm t}}_{#1}}

\nc{\BRC}{{\R}\hbox{-}{\mathrm{Cons}}}
\nc{\Brc}{{\R}\hbox{-}{\mathrm{C}}}

\newcommand{\indlim}[1][]{\mathop{\varinjlim}\limits_{#1}}

\newcommand{\prolim}[1][]{\mathop{\varprojlim}\limits_{#1}}

\nc{\eps}{\varepsilon}
\nc{\hs}{\hspace*}
\nc{\ms}{\mspace}
\nc{\nn}{\nonumber}
\nc{\tM}{\widetilde{M}}
\nc{\h}{\mathbf{h}}
\nc{\tf}{\tilde{f}}
\nc{\trf}{{{}^{\mathrm{t}}\mspace{-3mu}f}}
\nc{\codim}{\on{codim}}
\nc{\lh}{\mathscr{H}}
\nc{\bwr}{\scalebox{1.1}{$\wr$}}
\nc{\dTi}{\dT^{*,\mathrm{in}}}
\nc{\Cd}{\mathrm{C}}
\nc{\tK}{\widetilde{K}}
\nc{\aMM}{a_{M\times M}}
\nc{\e}{\mspace{1mu}\mathrm{e}\mspace{1mu}}
\nc{\lan}{\langle}
\nc{\ran}{\rangle}
\nc{\la}{\lambda}

\nc{\vphi}{\varphi}
\nc{\vep}{\varepsilon}

\nc{\At}{\tA_\twX}
\nc{\bu}{\boldsymbol{u}}
\nc{\bA}{\boldsymbol{A}}
\nc{\hbu}{\widehat{\boldsymbol{u}}}
\nc{\ex}{\mathrm{e}}
\nc{\vpi}{\varpi}

\nc{\one}{\mathbf{1}}
\nc{\setp}[1]{\{#1\}}
\nc{\GL}{\mathrm{GL}}
\nc{\cI}{\mathrm{I}}

\nc{\vs}{\vspace*}

\nc{\wb}[1]{\mbox{$\rule[-1.1ex]{0ex}{2ex}#1$}}
\nc{\wwb}[1]{\mbox{$\rule[-1.8ex]{0ex}{3ex}#1$}}
\nc{\bpi}{\ol{\pi}}
\nc{\Tsupp}{\Supp^{\mathrm E}}

\nc{\abu}{\Vec{\bu}}
\nc{\av}{\Vec{v}}
\nc{\au}{\Vec{u}}

\nc{\FN}{\mathrm{FN}}
\nc{\DFN}{\mathrm{DFN}}

\nc{\ake}{\hs{.25ex}}

\nc{\va}{\Vec{a}}
\nc{\bal}{\begin{align}}
\nc{\aal}{\end{align}}
\nc{\baln}{\begin{align*}}
\nc{\ealn}{\end{align*}}

\newcommand{\T}{\mathcal{T}}
\newcommand{\TP}{\mathcal{TP}}
\newcommand{\Stn}{\mathrm{Stn}}

\begin{document}

\title{Tempered subanalytic topology on algebraic varieties}
\author{Fran{\c c}ois Petit \footnote{The author has been fully supported in the frame of the OPEN scheme of the Fonds National de la Recherche (FNR) with the project QUANTMOD O13/570706}}

%\date{}

\maketitle

\begin{abstract}

On a smooth algebraic variety over $\C$, we build the tempered subanalytic and Stein tempered  subanalytic sites. We construct the sheaf of  holomorphic functions tempered  at infinity over these sites and study their relations with the sheaf of regular functions,  proving in particular that these  sheaves are isomorphic on Zariski open subsets. We show that these data allow to define the functors of tempered and Stein tempered analytifications. We study the relations between these two functors and the usual analytification functor.  
 We also obtain algebraization results in the non-proper case and flatness results.

\end{abstract}
\setcounter{tocdepth}{2}
\tableofcontents
\section*{Introduction}

The problem of comparing algebraic and complex analytic geometry is a very classical question. This comparaison is carried out by first functorially associating to an algebraic variety an complex analytic space with the help of the analytification functor. Then, one can try to compare the properties of these two spaces. In the case of proper algebraic varieties, this problem has been settle by Serre in his seminal paper \cite{Se55} and his famous GAGA theorem. In the non-proper case, this theorem does not hold but some comparison results between analytic and algebraic objects have been obtained under some growth condition on the analytic functions. One usually requires the holomorphic functions to be tempered which means that they have polynomial growth with respect to the inverse of the distance to the boundary. A very classical instance of such results is Liouville's theorem asserting that a tempered entire function is a polynomial. Several authors have studied tempered holomorphic functions on affine algebraic varieties (see for instance \cite{Bj74, Ru68, RW80}). Our aim, in this paper, is to sheafify the study of tempered holomorphic functions on algebraic variety and to compare the tempered geometry we obtain with algebraic geometry. For that purpose we construct two functors of tempered analytification, study their relations, properties and relate them to the usual analytification functor.
 
To carry out this idea, we rely upon Kashiwara-Schapira's theory of subanalytic sheaves \cite{KS01} which 
allows one to define sheaves the sections of which satisfy growth conditions (\cite{KS01}, \cite{SG16}). With the help of this theory, we associate to an algebraic variety the tempered subanalytic site, the Stein tempered subanalytic site and the sheaves of rings of tempered holomorphic functions over them. Then, we study the relationship between these sheaves and the sheaf of regular functions. We prove that on a Zariski open subset of $X$ the ring of tempered holomorphic functions is isomorphic to the ring of regular functions. This is a generalization of the aforementioned Liouville's Theorem since it implies that on a Zariski open set a tempered holomorphic function is a regular function.

In the last part of this paper, we discuss the possibility of getting a tempered GAGA theorem in the non-proper case. We construct an exact functor from the category of coherent algebraic sheaves to the category of modules over the sheaf of tempered holomorphic functions on the Stein tempered subanalytic site associated with $X$. 
We prove the fully faithfulness of this functor under the assumption of the vanishing of the tempered Dolbeaut cohomology on relatively compact subanalytic pseudo-convex open subset of $\C^n$. This vanishing result seems to be known to some specialists.\\

\noindent\textbf{Acknowledgment:} We would like to thank Pierre Schapira for many scientific advice and  for generously sharing is knowledge of the theory of subanalytic sheaves. We also thank Jean-Pierre Demailly for several pointed explanations concerning tempered Dolbeaut cohomolgy as well as Mauro Porta for useful discussions. We also thanks Tony Pantev for his invitation and hospitality to the University of Pennsylvania where this work was completed.

\section{Sheaves on the subanalytic site}
The definitions and results of the subsections \ref{subsec:satop} and \ref{subsec:temphol} are due to~\cites{KS96, KS01} and we follow closely their presentation. We also refer to \cite{Pre08}, where the operations for subanalytic sheaves are developed without relying on the theory of ind-sheaves.
\subsection{Subanalytic topology, tempered functions and distributions: a review}\label{subsec:satop}
\subsubsection{The subanalytic topology}
\begin{definition}
\banum
\item Let $M$ be a real analytic manifold. Let $\Op_M$ be the category where the objects are the open subsets of $M$ and the morphisms are the inclusions between open subsets. The category $\Op_\Msa$ is the full subcategory of $\Op_M$ spanned by the relatively compact subanalytic open subsets of $M$. 

\item The site $\Msa$ is the category $\Op_\Msa$ endowed with the Grothendieck topology where a family  $\{U_i\}_{i\in I}$ of objects is a covering of $U \in\Op_{\Msa}$ if for every $i \in I$, $U_i \subset U$ and there exists a finite subset $J\subset I$ such that  $\bigcup_{j\in J}U_j=U$.
\eanum
\end{definition}

Let $\cor$ be a field. The Grothendieck category of sheaves of $\cor$-modules on $\Msa$ is denoted by $\md[\cor_\Msa]$.
The next result is useful to construct subanalytic sheaves. It is a special case of \cite[Proposition 6.4.1]{KS01}. 
\begin{proposition}\label{prop:MVsa1} A presheaf $F$ on $\Msa$ is a sheaf if and only if $F(\emptyset)=0$ and for any pair $(U_1,U_2)$ in $\Op_\Msa$, the sequence
\eqn
0\to F(U_1\cup U_2)\to F(U_1)\oplus F(U_2)\to F(U_1\cap U_2)
\eneqn 
 is exact.
\end{proposition}

The morphism of sites $\rhosa$ induces the following adjoint pair of functors
\eq\label{eq:fctrho0}
\xymatrix{
\opb{\rhosa} \colon \md[\cor_\Msa] \ar@<.5ex>[r]&\md[\cor_M]\ar@<.5ex>[l]\colon \oim{\rhosa}\,.
}
\eneq
Moreover, $\opb{\rhosa}$ is exact and the functor $\oim{\rhosa}$ is fully faithful and left exact. This gives rise to the adjoint pair of functors
\eq\label{eq:fctrho1}
\xymatrix{
 \opb{\rhosa} \cl \Derb(\cor_\Msa) \ar@<.5ex>[r]&\Derb(\cor_M)\ar@<.5ex>[l] \cl \roim{\rhosa}\,.
}
\eneq
The functor $\roim{\rhosa}$ is fully faithful.

The functor $\opb{\rhosa}$ has also a left adjoint, denoted $\eim{\rhosa} \cl \md[\cor_M] \to \md[\cor_\Msa]$. We refer the reader to \cite[\S 6.6]{KS01} for more details and only recall its construction and main properties. If  $F\in \md[\cor_M]$, $\eim{\rhosa}F$ is defined as the sheafification of the presheaf $\Op_{\Msa} \ni U\mapsto F(\overline U)$. The functor $\eim{\rhosa}$
is exact, fully faithful and commutes with tensor products.

The situation is summarized by the following diagram.
\eqn
&& \xymatrix{
{\md[\cor_M]}\ar@<-1.5ex>[rr]_-{\eim{\rhosa}}\ar@<1.5ex>[rr]^-{\oim{\rhosa}}&&{\md[\cor_\Msa]}\ar[ll]|-{\;\opb{\rhosa}\;}.
}
\eneqn
where $(\opb{\rhosa},\oim{\rhosa})$ and $(\eim{\rhosa},\opb{\rhosa})$ form adjoint pairs of functors.

Note that $M$ is not an object of $\Op_\Msa$ unless $M$ is compact.  Thus, we define the global section functor as follow. For an object $F$ of $\Der(\cor_\Msa)$, one sets
 \eq\label{eq:rsectMsa}
 &&\rsect(\Msa;F)=\RHom(\cor_\Msa,F).
 \eneq
 It follows from the isomorphism $\cor_\Msa\simeq\eim{\rhosa}\cor_M$, the adjunction 
 $(\eim{\rhosa},\opb{\rhosa})$ and formula~\eqref{eq:rsectMsa} that
\eq\label{eq:rsectMsa2}
&&\rsect(\Msa;F)\simeq\rsect(M;\opb{\rhosa} F).
\eneq

\subsubsection{Induced topology}
Let $U\in\Op_\Msa$.  We endow $U$ with the topology induced by $\Msa$, that is,
 we consider the site $\Umsa$ defined as follows
 \banum
\item
the objects of $ \Op_{\Umsa}$  are the open subanalytic subsets of $U$, no more necessarily relatively compact,
\item
a covering of $V \in \Op_{\Umsa}$  is a family $\{V_i\}_{i \in I}$ of objects of $\Op_\Msa$ such that $V_i\subset V$ and there is a finite subset $J\subset I$ with $V=\bigcup_{j\in J}V_j$. 
\eanum
Note that the natural morphism of sites $\Umsa \cl \Usa \to \Umsa$ is not an equivalence of sites in general.

\subsubsection{Tempered functions and distributions}\label{subsec:temp}
All along this paper $M$ is a real analytic manifold. We denote by $\Cinf$ the sheaf of  $\C$-valued functions of class $\mathrm{C}^\infty$ on $M$ and by $\Db_M$ the sheaf of Schwartz's distributions on $M$.

\begin{definition}\label{def:polgrowth}
Let $U$ be an open subset of $M$  and $f\in \Cinf(U)$. 
The function $f$ has {\it  polynomial growth} at $p\in M$
if  $f$ satisfies the following condition: for a local coordinate system $(x_1,\dots,x_n)$ around $p$, there exist a sufficiently small compact neighbourhood $K$ of $p$
and a positive integer $N$ such that
\eq
&&\sup\limits_{x\in K\cap U}\big(\dist(x,K\setminus U)^N\vert f(x)\vert )\big)
<\infty\,.\label{eq:moderate}
\eneq
Here, $\dist(x,K\setminus U)\seteq\inf\set{|y-x|}{y\in K\setminus U}$,
and we understand that the left-hand side of \eqref{eq:moderate} is $0$
if $ K\cap U=\emptyset$ or $K\setminus U=\emptyset$. We say that $f$ has {\em polynomial growth} 
if it has polynomial growth at any point of $M$. We say that $f$ is {\em tempered at $p$ } if all its derivatives have polynomial growth at $p$. We say that $f$ is {\em tempered} 
if it is tempered at any point of $M$.
\end{definition}

\begin{remark}\label{rem:polyimptem}
\noindent (i) In the above definition, $f$ has polynomial growth at any point of the open subset $U$.

\noindent (ii) A holomorphic function which has polynomial growth on a relatively compact open subset $U$ of $\C^n$ is tempered on $U$. This follows from the Cauchy formula (see \cite[Lemma 3]{Si70} for a detailed proof). 

\end{remark}

We recall Lojaciewicz's inequality.

\begin{theorem}[Lojaciewicz's inequality \cite{BM88}] Let $M$ be a real analytic manifold and let $K$ be a subanalytic subset of $M$. Let $f, \, g\colon K \to \R$ be subanalytic functions with compact graphs. If $f^{-1}(0) \subset g^{-1}(0)$, then there exists $c, \, r >0$ such that, for all $x \in K$,
\eqn
\vert f(x) \vert \geq c \, \vert g(x) \vert^r.
\eneqn
\end{theorem}

The following metric property of subanalytic subsets of $\R^n$ is a consequence of the Lojaciewicz's inequality (See the seminal papers by Lojasiewicz~\cite{Lo59} and Malgrange~\cite{Ma66}). This result is used to construct the sheaf of tempered smooth functions.

\begin{lemma}\label{th:Loj1}
Let $U$ and $V$ be two relatively compact subanalytic open  subsets of $\R^n$.
Then, there exist a positive integer $N$ and a constant $C>0$ such that
\eqn
&&\dist\big(x,\R^n\setminus (U\cup V)\big)^N
\leq C \big(\dist(x,\R^n\setminus U)+\dist(x,\R^n\setminus V)\big).
\eneqn
\end{lemma}

We define the presheaf $\Cinft[\Msa]$ by $\Op_{\Msa} \ni U \mapsto \Cinft[M](U)$ where $\Cinft[M](U)$ is the subspace of $\Cinf(U)$ consisting of tempered $\mathrm{C}^\infty$-functions. We also consider the presheaf $\Dbt_{\Msa}$ which associates to a subanalytic open subset $U$ of $M$ the image $\Dbt_M(U)$ 
of the restriction morphism $\sect(M;\Db_M)\to\sect(U;\Db_M)$. The space $\Dbt_M(U)$ is called the space of {\em tempered distributions} on $U$.

\begin{proposition}
The presheaves  $U\mapsto \Cinft[M](U)$ and $U\mapsto \Dbt_M(U)$ are sheaves on  $\Msa$.
\end{proposition}

\begin{proof}
This follows from Lemma \ref{th:Loj1} and Proposition \ref{prop:MVsa1}.
\end{proof}

We have the following natural morphisms (see \cite[p.122]{KS01})
\eq
&&\eim{\rhosa}\Cinf[M]\to\Cinft[\Msa]\to\oim{\rhosa}\Cinf[M]
\eneq
and isomorphisms
\eq\label{eq:opbCtm}
&&\opb{\rhosa}\eim{\rhosa}\Cinf[M]\isoto\opb{\rhosa}\Cinft[\Msa] \isoto\opb{\rhosa}\oim{\rhosa}\Cinf[M]
\isoto\Cinf[M],\quad\opb{\rhosa}\Dbt_{\Msa}\simeq\Db_M.
\eneq

\subsection{Tempered functions and distributions: complements}

In this subsection, we establish some complementary results concerning tempered functions and distributions. All manifolds $M_1$, $M_2$ etc. are real analytic.

Set for short $M_{ij}=M_i\times M_j$, $M_{123}=M_1\times M_2\times M_3$ and 
denote by $p_i$ the $i$-th projection defined on $M_{ij}$ or on $M_{123}$ and by $p_{ij}$ the $(i,j)$ projection.

Recall that if  $K$ is a closed subanalytic subset of $M_1$, then, a function $f \colon K \to M_2$ is subanalytic if its graph  $\Gamma_f$ is subanalytic in $M_{12}$.

\begin{lemma} \label{lem:compsub} 
Let $M_1$, $M_2$ and $M_3$ be real analytic manifolds, $K$ be a closed subanalytic subset of $M_1$, $f \colon K \to M_2$ and $g\colon M_2 \to M_3$ be continuous subanalytic functions. Then $g \circ f$ is a subanalytic function.
\end{lemma}

\begin{proof}
By assumption  $\Gamma_f$ is a closed subanalytic  subset of $M_{12}$
and $\Gamma_g$  is a closed subanalytic subset of $M_{23}$. The set $\opb{p}_{23}(\Gamma_g) \cap \opb{p}_{12}(\Gamma_f)$ is subanalytic and closed in $M_{123}$. Since 
%For $K_{ij}$ a compact subset of $M_i \times M_j$ \ber($ij=12$ or $ij=23$) \er the set $\opb{p}_{12}(K_{12}) \cap \opb{p}_{23}(K_{23})$ \cmt{23 pas 13} is compact. This implies that 
$p_{13}$ is proper over $\opb{p}_{23}(\Gamma_g) \cap \opb{p}_{12}(\Gamma_f)$, the set
\eqn
\Gamma_{g \circ f} = p_{13}(\opb{p}_{23}(\Gamma_g) \cap \opb{p}_{12}(\Gamma_f)).
\eneqn

is subanalytic in $M_{13}$.
\end{proof}

We recall the following fact.

\begin{lemma}[{\cite[Remark 3.11]{BM88}}] \label{lem:funsub} 
Let $A$ be a subanalytic subset of $\R^n$. Then, the function $\R^n \to \R$, $x \mapsto \dist(x,A)$ is subanalytic.
\end{lemma}
Here, $\dist$ is the Euclidian distance.

The following Proposition is probably well-known but we have no reference for it. So, we provide it with a proof.
\begin{proposition}\label{lem:precomposition}
Let $M$ and $N$ be two real analytic manifolds, let $f\cl M \to N$ be a real analytic map, let $V$ be a subanalytic open subset of $N$ and set $U=f^{-1}(V)$. Let $\psi$ be a smooth function with polynomial growth on $V$. Then $\psi \circ f$ has polynomial growth on $U$. 
\end{proposition}

\begin{proof}
The preimage of a subanalytic set by a real analytic map is subanalytic, thus $U$ is subanalytic in $M$. Let $p \in M$. In view of Definition \ref{def:polgrowth} and Remark \ref{rem:polyimptem}, we can assume that $p \in \partial U$. There is a coordinates system  $(y_1,\ldots,y_m)$ around $f(p)$ as well as a sufficiently small subanalytic compact neighbourhood $K^\prime$ of $f(p)$ and a non-negative integer $N^\prime$ such that
\begin{equation*}
\sup\limits_{y\in K^\prime\cap V}\big(\dist(y,K^\prime\setminus V)^{N^\prime}\vert \psi(y)\vert\big)
<\infty.
\end{equation*}
There is a coordinates system $(x_1,\ldots,x_m)$ around $p$ and a sufficiently small subanalytic compact neighbourhood $K$ of $p$ such that $\psi(K) \subset K^\prime$. Shrinking $K$ and $K^\prime$ if necessary, we can further assume that 

 $\dist(x,K\setminus U) \leq 1$  and $\dist(f(x),K^\prime \setminus V) \leq 1$ for $x\in K$.

The function $f:K \to N$ is subanalytic in $M \times N$. It follows from Lemmas~\ref{lem:funsub} and~\ref{lem:compsub} that the continuous functions on $K$, $\dist(x,K\setminus U)$ and $\dist(f(x),K^\prime \setminus V)$ are subanalytic.

Moreover,  $\dist(f(x),K^\prime \setminus V)=0$ implies $f(x)\notin V$ hence $x\notin U$ and thus $\dist(x,K\setminus U)=0$. 
We deduce from the Lojaciewicz's inequality that there exists $C, \, \alpha \in \R_+^\ast$ such that for every $x \in K$
\begin{equation*}
\dist(x,K\setminus U)^\alpha\leq C \,\dist(f(x),K^\prime \setminus V).
\end{equation*} 
There is a non-negative integer $N \geq \max(\alpha N^\prime, N^\prime)$ such that for every $x \in K \cap U$
\begin{equation*}
\dist(x,K\setminus U)^N |\psi \circ f(x)| \leq \dist(x,K\setminus U)^{\alpha N^\prime} |\psi\circ f(x)| \leq C \,\dist(f(x),K^\prime \setminus V)^{N^\prime} |\psi\circ f(x)|
\end{equation*}
which proves that $\sup\limits_{x\in K\cap U}\big(\dist(x,K\setminus U)^N\vert \psi\circ f(x)\vert \big)<\infty$.
\end{proof}

\begin{corollary}\label{cor:tempprecompo}
Let $M$ and $N$ be two real analytic manifolds, let $f\cl M \to N$ be a real analytic map, let $V$ be a subanalytic open subset of $N$ and set $U=f^{-1}(V)$. Let $\psi$ be tempered smooth function on $V$. Then $\psi \circ f$ is a tempered smooth function on $U$. 
\end{corollary}

Finally, there is the following important result concerning $\Dbt_{Msa}$. This is essentially a corollary of Lemma 2.5.7 of \cite{KS16} but it has never been stated explicitly and due to its importance we provide a detailed proof.

Consider a morphism of real analytic manifold $f\cl M\to N$. Recall the natural isomorphism \cite{KS16}*{Th.~2.5.6}
\eq\label{eq:KS256}
&&\Dbtv_\Msa\ltens[\shd_\Msa]\shd_{\Msa\to\Nsa}\isoto\epb{f}\Dbtv_\Nsa.
\eneq
Recall that $\Dbtv_\Msa$ is the tensor product of  $\Dbt_\Msa$ with the sheaf of analytic densities (differential form of higher degree tensorized with the orientation sheaf).

\begin{proposition}\label{prop:Dbinvariance}
Consider a morphism of real analytic manifolds $f\cl M\to N$ and let $U \in \Op_{\Msa}$ and $V \in \Op_{\Nsa}$. Assume that $f$ induces an isomorphism $U \isoto V$. Then 
\banum
\item
$f\vert_U$ induces an isomorphism of sites $\Umsa \simeq\Vnsa$,
\item
isomorphism~\eqref{eq:KS256} induces an isomorphism 
\eq\label{eq:KS256b}
\oim{f\vert_U}\Dbt_{\Msa}\vert_\Umsa &\simeq &\Dbt_{\Nsa}\vert_\Vnsa.  
\eneq
\eanum
\end{proposition}

\begin{proof}
(a)  is clear.

\spa
(b) 
It is enough to prove that for $U'\in\Op_\Umsa$ and $V'\in\Op_\Vnsa$ with $f(U')=V'$, we have an isomorphism
\eqn
\Dbt_{\Msa}(U')&\simeq &\Dbt_{\Nsa}(V'). 
\eneqn
Changing our notations, we will prove this formula for $U$ and $V$. 

\spa
To prove this statement, we work with indsheaves. This does not cause any problem since subanalytic sheaves form a full subcategory of the category of indsheaves (see Proposition 2.4.3 of \cite{KS16} for more details). We follow the notations of \cite{KS16} for all matters regarding ind-sheaves. We denote by $\Dbt_M$  and $\Dbt_N$ the subanalytic sheaves 
$\Dbt_\Msa$  and  $\Dbt_\Nsa$ viewed as  indsheaves.
In particular, we have
\eqn
&&\Dbt_{\Msa}(U) \simeq \RHom[\Der(\icor_M)](\C_U,\Dbt_{M})
\eneqn
and similarly with  $\Dbt_\Nsa(V)$.
From Lemma 2.5.7 of \cite{KS16}, we have
\eqn
\rihom(\C_U,\Dbt_{M})&\simeq &\epb{f}\rihom(\C_V,\Dbt_{N}).
\eneqn
Then,
\eqn
\RHom[\Der(\icor_M)](\C_U,\Dbt_{M})&\simeq&\RHom[\Der(\icor_M)](\C_U,\rihom(\C_U,\Dbt_{M}))\\
&\simeq &\RHom[\Der(\icor_M)](\C_U,\epb{f}\rihom(\C_V,\Dbt_{N}))\\
&\simeq &\RHom[\Der(\icor_M)](\C_V \tens \reiim{f} \C_U,\Dbt_{N})\\
&\simeq &\RHom[\Der(\icor_M)](\C_V \tens \reim{f} \C_U ,\Dbt_{N})\\                                  
                                &\simeq &\RHom[\Der(\icor_M)](\C_V,\Dbt_{N}).
\eneqn
\end{proof}

\begin{remark}\label{rem:ringstruct}
It follows from Corollary \ref{cor:tempprecompo} that there is a morphism of sheaves of rings
\eq\label{mor:Cinftemp}
\Cinft[\Nsa] \to \oim{f}\Cinft[\Msa], \; \varphi \mapsto \varphi \circ f.  
\eneq
It follows from the commutativity of  Diagram~\eqref{diag:Dbinvariance} below
\eq\label{diag:Dbinvariance}
\xymatrix{\eim{f}\Dbv_{U} \ar[r]^-{\int_f} & \Dbv_{V}\\
          \eim{f}\Db_{U} \ar[u]^-{\cdot \otimes d\lambda}_-{\wr}          &  \Db_{V}\ar[u]_-{\cdot \otimes d\lambda}^-{\wr}\\
\oim{f_{U}} \mathcal{C}^\infty_{U} \ar[u]^-{\int(\cdot)|Jac(f)|d\lambda}  & \mathcal{C}^\infty_{V} \ar[u]_-{\int(\cdot)d\lambda} \ar[l]_-{f^\ast}          
}
\eneq
that the isomorphism~\eqref{eq:KS256b} is compatible with the morphism \eqref{mor:Cinftemp}. This shows that \eqref{mor:Cinftemp} induces an isomorphism of sheaves of rings
\eq\label{eq:KS256c}
\oim{f\vert_U}\Cinft[\Msa]\vert_\Umsa &\simeq &\Cinft[\Nsa]\vert_\Vnsa.  
\eneq
\end{remark}

\subsection{Tempered holomorphic functions}\label{subsec:temphol}

In this subsection, we construct the sheaf of tempered holomorphic functions.
Let $X$ be a complex manifold. We denote  by $X^c$ the complex conjugate manifold of $X$ and by $X_\R$ the underlying real analytic  manifold. We write $\Xsa$ for the manifold  $X_\R$ endowed with the subanalytic topology and set
\eqn
&&\D_\Xsa\eqdot\erx\shd_X.
\eneqn

Following \cite{KS01}, we define the following sheaves on $\Xsa$.
\eq
\Owx&\eqdot&\eim{\rhosa}\sho_X,\label{eq:defOt1}\\
\Otx&\eqdot&\rhom[\shd_{\Xcsa}](\Ocwx,\Dbt_{X_\R}),\label{eq:defOt2}\\
\Ox&\eqdot&\roim{\rhosa}\sho_X.\label{eq:defOt4}
\eneq
The objects $\Owx, \Otx$ and $\Ox$ belong to the category $\Derb(\D_\Xsa)$.  
The object $\Otx$ is isomorphic to the Dolbeault complex with coefficients in $\Dbt_{X_\R}$ on the subanalytic site:
\eq\label{eq:dolbeault1}
&&0\To\Dbt_{\Xsa}\To[\ol\partial]\Db_{\Xsa}^{\mathrm{t}\;(0,1)}\To[\ol\partial]\cdots
\To[\ol\partial]\Db_{\Xsa}^{\mathrm{t}\;(0,d_X)}\To0.
\eneq

One calls  $\Otx$  the 
{\em sheaf of tempered holomorphic functions}. 
We have natural morphisms in  $\Derb(\D_\Xsa)$:
\eqn
&& \Owx\to \Otx\to\Ox.
\eneqn
By \cite[Theorem 10.5]{KS96} we have the isomorphism
\eq\label{eq:defOt5}
\Otx&\simeq&\rhom[\shd_{\Xcsa}](\erx\sho_{X^c},\Cinft[\Xsa])\mbox{ in }\Derb(\D_\Xsa).
\eneq
This implies, in particular, that $\Otx$ is represented by the differential graded algebra
\eq \label{eq:dolbeault2}
&&0\To\Cinfn{\Xsa}{\mathrm{t}}\To[\ol\partial]\Cinfn{\Xsa}{\mathrm{t}\;(0,1)}\To[\ol\partial]\cdots
\To[\ol\partial]\Cinfn{\Xsa}{\mathrm{t}\;(0,d_X)}\To0.
\eneq

Consider a morphism of complex manifolds $f\cl X\to Y$ 
 First, recall the morphism of~\cite{KS01}*{Lem.~7.4.9} or~\cite{KS16}*{Cor.~3.1.4}
\eq\label{eq:KS314}
&&\shd_{\Xsa\to \Ysa}\ltens[\opb{f}\shd_\Ysa]\opb{f}\Ot[\Ysa]\to\Ot[\Ysa]
\eneq
Note that this morphism is constructed from morphism \eqref{eq:KS256}.

\begin{proposition}\label{prop:morphcompact}
Consider a morphism of complex manifolds $f\cl X\to Y$ and let $U \in \Op_{\Xsa}$ and $V \in \Op_{\Ysa}$. Assume that $f$ induces an isomorphism of complex analytic manifolds $U \isoto V$. Then 
isomorphism~\eqref{eq:KS256} induces an isomorphism 
\eq\label{eq:KS314b} 
\roim{f\vert_U}\Ot[\Xsa]\vert_\Uxsa &\simeq &\Ot[\Ysa]\vert_\Vysa.  
\eneq
\end{proposition}

\begin{proof}
We shall use Proposition~\ref{prop:Dbinvariance}. 
 Recall first that $f\vert_U$ induces an isomorphism of sites $\Uxsa \simeq\Vysa$ and that the morphism of $\shd_{\Xsa}$-module $\shd_{\Xsa} \to \shd_{\Xsa\to \Ysa}$ restricted to $\Uxsa$ is an isomorphism. 
We have the sequence of isomorphisms
\eqn
\oim{f\vert_U}\Ot[\Xsa]\vert_\Uxsa&\simeq&\oim{f\vert_U}\rhom[\D_{\Xcsa}](\Ocwx, \Dbt_\Xsa)\vert_\Uxsa\\
&\simeq&\rhom[\D_{\Ycsa}](\Ocwy, \Dbt_\Ysa)\vert_\Vysa\\
&\simeq&\Ot[\Ysa]\vert_\Vysa.
\eneqn
\end{proof}

\begin{remark}\label{rem:ringstructsa}
The sheaf $\rmH^0(\Otx)$ is a sheaf of rings whose multipicative law is given by the multiplication of functions. Isomorphism~\eqref{eq:KS314b} induces an isomorphism of sheaves of rings:
\eq\label{eq:KS314ring} 
\oim{f\vert_U}\rmH^0(\Otx)\vert_\Uxsa &\simeq &\rmH^0(\Oty)\vert_\Vysa.  
\eneq
This follows from the isomorphism~\eqref{eq:KS256c}.
\end{remark}

The following result seems known to some specialists but there is no reference for it in the literature so we do not label it as a theorem.
\begin{conjecture}\label{claim:dolbeautvan}
Assume that $X= \C^n$ and that $U$ is a relatively compact subanalytic pseudo-convex open subset of $X$. 
Then $\rsect(U;\Otx)\in\Derb(\C)$  is concentrated in degree $0$.
\end{conjecture}

\section{Tempered analytification}\label{section:tempered}

\subsection{$\TP$-spaces}

In this subsection, we review the notion of $\TP$-space. This a notion which is intermediate between the notion of topological space and the concept of site and will be better suited to our purpose than the notion of ringed sites. It was introduced in \cite{KS01} and further studied in \cite{PrEd10}.

\begin{definition}
\begin{nnum}
\item
A $\TP$-space $(X,  \T)$ is the data of a set $X$ together with a collection of subsets $\T$ of $X$ such that\\
(i) $\emptyset \in \T$,
(ii) $\T$ is stable by finite unions and finite intersections.
\item
A morphism of $\TP$-spaces $f\cl (X, \T_X)\to(Y,\T_Y)$ is a map $f \cl X \to Y$ 
such that for every $U \in \T_Y$, $\opb{f}(U) \in \T_X$.
\end{nnum}
\end{definition}
We have thus defined the category of $\TP$-spaces, that we denote by $\mathsf{TS}$.

The forgetful functor $for \cl \mathsf{Top} \to \mathsf{TS}$ form the category of topological spaces to the category of $\TP$-spaces admits a left adjoint $(\cdot)^{\mathrm{Top}}:\mathsf{TS} \to \mathsf{Top}$ which associate to a $\T$-space $(X,\T)$ a topological space $(X,\T^\mathrm{Top})$ where $\T^\mathrm{Top}:=\{\bigcup_{V \in \mathcal{B}} V \vert \mathcal{B} \subset \T \}$. When there is no risk of confusion we will write $X^\mathrm{Top}$ instead of $(X,\T^\mathrm{Top})$.

A $\TP$-space defines a presite, the morphisms $U\to V$ being the inclusions. We  endow it with the following Grothendieck topology: A family  $\{U_i\}_{i\in I}$ of objects of $\T$ is a covering of $U$ if for every $i \in I$, $U_i \subset U$ and there exists a finite subset $J\subset I$ such that  $\bigcup_{j\in J}U_j=U$. 

We denote this site by $X_\T$. 
A morphism of $\TP$-spaces $f\cl (X, \T_X) \To (Y,\T_Y)$ induces a morphism of sites $f \cl X_{\T_X} \to Y_{\T_Y}$, $\T_Y \ni V \mapsto \opb{f}(V) \in \T_X$. 
\begin{notation}
If there is no risk of confusion, we  will not make the distinction between a $\TP$-space $(X,\T)$ and the associated site $X_\T$. 
\end{notation}
Given  a $\TP$-space $X$ and $U\in \T$, $U$ is naturally endowed with a structure of $\TP$-space $(U,\T_U)$ by setting 
$\T_U=\{ V \subset U; V \in \T\}$. We denote by $U_{X_\T}$ the site induced by $X_\T$ on the presite $(U,\T_U)$. Hence, the coverings of $U_{X_\T}$ are those induced by the coverings in $X$.

\begin{definition} 
A site $X_\T$ (associated with a $\TP$-space) endowed with a sheaf of rings $\sho_X$ is called a ringed $\TP$-space. One defines as usual a morphism of ringed $\TP$-spaces .
\end{definition}

\subsection{The tempered analytification functor}
The aim of this subsection is to construct the tempered analytification functor.

\begin{notation} We denote by
\begin{itemize}
\item $\mathsf{Var}$ (resp. $\mathsf{Var_{sm}}$) the category of complex algebraic varieties (resp. smooth complex algebraic varieties),
\item $\mathsf{An_\C}$ the category of complex analytic spaces,
\item $(\cdot)^\an:\mathsf{Var}\to \mathsf{An}_\C$ the analytification functor,
\item $\mathsf{TRgS}$ the category of ringed $\TP$-spaces.

\item $\mathsf{TFS}$ the category whose objects are the pair $((X,\T),F)$ where $(X, \T)$ is a $\mathcal{TP}$-space and $F \in \Derb(\cor_{X_\T})$. A morphism in $\mathsf{TFS}$  is a pair $(f,f^\sharp)\cl ((X,\T_X),F) \to ((Y,\T_Y),G)$ where $f:X \to Y$ is a morphism of $\mathcal{TP}$-space and $f^\sharp \cl \opb{f}G \to F$ is a morphism in $\Derb(\cor_{X_\T})$.
\end{itemize}

In general, if there is no risk of confusion, we shall not write the functor $(\cdot)^\an$. For example, if $X$ is a smooth complex algebraic varieties, then $\Xsa$ is the subanalytic site associated with $X^\an$. 
\end{notation}

We refer the reader to \cite[Expos\'e XII]{SGA1} for a detailed study of the properties of the analytification functor.  

We recall a few classical facts concerning smooth compactifications of algebraic varieties. 

\begin{enumerate}[(i)]
\item It follows from the Nagata's compactification theorem and Hironaka's desingularization theorem that any smooth algebraic variety has a smooth algebraic compactification.

\item Compactification is not functorial but let $f: X_0 \to X_1$ be a  morphism of algebraic varieties and $j_1 \cl X_1 \to Y_1$ be a compactification of $X_1$. Then, there exists a morphism $\hat{f}:Y_0 \to Y_1$ where $j_0 \cl X_0 \to Y_0$ is a smooth compactification of $X_0$ such that the following diagram commutes
\eq\label{diag:compmap}
\xymatrix{ Y_0 \ar[r]^-{\hat f}& Y_1\\
X_0 \ar@{_{(}->}[u]^{j_0} \ar[r]^-{f} & X_1 \ar@{_{(}->}[u]_-{j_1}.
}
\eneq

This is a consequence of the following classical construction. 

Let $f \cl X_0 \to X_1$  be a morphism of algebraic varieties. Let $Y_2$ and $Y_1$ be respectively smooth compactification of $X_0$ and $X_1$. Consider $\Gamma_f$ the graph of $f$ as a subset of $Y_2 \times Y_1$ and consider its closure, in the Zariski topology, $\ol{\Gamma_f}^{\,\zar}$. Note that $\Gamma_f$ is open in $\ol{\Gamma_f}^{\,\zar}$. By Hironaka's theorem there exists a smooth algebraic variety $Y_0$ and a proper morphism $g \cl Y_0 \to \ol{\Gamma_f}^{\,\zar}$, such that $g: g^{-1}(\Gamma_f) \to \Gamma_f$ is an isomorphism. This implies that there is an open embedding $j_0: X_0 \to Y_0$. Thus $j_0$ is a smooth compactification of $X$. 

Let $p_i \cl Y_2 \times Y_1 \to Y_i$ $(i=1,\,2)$ be the projection on $Y_i$ and we still denote by $p_i$ their respective restrcition to $\ol{\Gamma_f}^{\,\zar}$. We set $\widehat{f}=p_1 \circ g$ and notice that  $\widehat{f}|_{X_0}=f$.
Then, we have the diagram
\begin{equation}\label{diag:compactification}
\xymatrix{
Y_2  & \ol{\Gamma_f}^{\,\zar} \ar[r]^-{p_1} \ar[l]_-{p_2}& Y_1\\
&  Y_0 \ar[u]^-{g}\ar[ru]_-{\widehat f}&\\
X_0 \ar@{=}[r] \ar@{_{(}->}[uu]^{j_2}& X_0 \ar@{_{(}->}[u]^{j_0}  \ar[r]^f & X_1, \ar@{_{(}->}[uu]^{j_1}\\
}
\end{equation}

\item Smooth compactifications are not unique but there is the following consequence of the construction in (ii). Consider two smooth compactifications of $X$, $j_1\cl X \hookrightarrow Y_1$ and $j_2 \cl X \hookrightarrow Y_2$.  Applying (ii) with $X_0=X_1=X$ and $f=\id$, we get the following commutative diagram

\eq\label{diag:dominationcomp}
\xymatrix{&Y_0\ar[rd]^-{q_2} \ar[ld]_-{q_1}&\\
Y_1& \ar[r]_-{j_2} \ar[l]^-{j_1} \ar[u]^-{j_0}X& Y_2
}
\eneq
where $Y_0$ is a smooth compactification of $X$, $q_1=p_1 \circ g$, $q_2=p_2 \circ g$. The morphism $q_1$  (resp. $q_2$) induces an isomorphism $q_1 \cl j_0(X) \isoto j_1(X)$ (resp. $q_2 \cl j_0(X) \isoto j_2 (X)$).
\end{enumerate}

\begin{lemma}\label{lem:indep}
Let $X$ be a smooth algebraic variety and $Y$ a smooth algebraic compactification of $X$.
\banum
\item The site $X_{Y_\sa}$ does not depend on the choice of a smooth algebraic compactification $Y$ of $X$,

\item the object $\Ot[Y_{\sa}]\vert_{X_\Ysa} \in \Derb(\cor_{X_\Ysa})$ does not depend on the choice of an algebraic compactification $Y$ of $X$.
\eanum
\end{lemma}

\begin{proof}
\banum
\item is a consequence of Proposition \ref{prop:morphcompact} (a) and Diagram \eqref{diag:dominationcomp}.
\item is a consequence of Proposition \ref{prop:morphcompact} (b) and Diagram \eqref{diag:dominationcomp}.
\eanum
\end{proof}
As a consequence of Lemma \ref{lem:indep} we can state
\begin{definition}
Let $X$ be a smooth algebraic variety. We define  the \textit{tempered analytification} (t-analytification for short) of $X$ to be the object of $\mathsf{TFS}$, $(X^\temp,\Otempx)$ where
\eqn
&& X^\temp \eqdot (X(\C), \, X_{Y_\sa})  \quad \textnormal{and} \quad \Otempx \eqdot\Ot[Y_\sa]|_{X_{Y_\sa}}.
\eneqn
\end{definition}

We will associate to a morphism of smooth algebraic varieties $f \cl X_0 \to X_1$ is tempered $f^\temp$ analytification. In order to do so, we need the following Lemma.

\begin{lemma}\label{lem:morwelldef}
Let $f \cl X_0 \to X_1$ be a morphism of smooth algebraic varieties. Let $U \in \Op_{X_1^\temp}$. Then $(f^\an)^{-1}(U) \in \Op_{X_0^\temp}$.
\end{lemma}

\begin{proof}
There exist smooth algebraic compactifications $Y_0$ and $Y_1$ of $X_0$ and $X_1$ such that $f$ extends to a regular morphism $\widehat{f} \cl Y_0 \to  Y_1$ such that Diagram \eqref{diag:compmap} commutes. Let $U \in \Op_{X^\temp}$, then
\begin{align*}
f^{-1}(U)&=j_0^{-1}\widehat{f}^{-1}(j_1(U))\\
                     &=\widehat{f}^{-1}(U) \cap X_0^\an.
\end{align*}
It follows that $f^{-1}(U)$ is subanalytic in $Y_0$.
\end{proof}
Let $f \cl X_0 \to X_1$ be a morphism of smooth algebraic variety. Using Lemma \ref{lem:morwelldef}, we associate to $f$ a morphism of site 
\eq\label{mor:ftemp}
f^\temp \cl X_0^\temp \to X_1^\temp, \;\Op_{X_1^\temp} \ni U \mapsto (f^\an)^{-1}(U) \in \Op_{X_0^\temp}.
\eneq 
\begin{lemma}\label{lem:morsheaftemp}
Let $f \cl X_0 \to X_1$ be a morphism of smooth algebraic varieties. Then $f$ induces a canonical morphism in $\Derb(\cor_{X_0^\temp})$
\eqn
f^{\temp \,\sharp} \cl \opb{(f^\temp)}\Ot[X_1^\temp] \to \Ot[X_0^\temp]
\eneqn
\end{lemma}

\begin{proof}
There exist smooth algebraic compactifications $Y_0$ and $Y_1$ of $X_0$ and $X_1$ such that $f$ extends to a regular morphism $\widehat{f} \cl Y_0 \to  Y_1$ such that Diagram \eqref{diag:compmap} commutes. Using the presentation of $\Ot[Y_{0\,\sa}]$ (resp. $\Ot[Y_{1\sa}]$) by the complexe \eqref{eq:dolbeault2} and Corollary \ref{cor:tempprecompo}, the pullback of differential forms by $f$ provides a morphism
\eqn
(\Cinfn{Y_{1\sa}}{\mathrm{t}\;(0,\bullet)},\; \ol\partial) \to \oim{\hat f}(\Cinfn{Y_{0\sa}}{\mathrm{t}\;(0,\bullet)},\;\ol\partial)) ;\quad \omega \mapsto {\hat f}^\ast \omega
\eneqn
Restricting to $X^\temp$ and using the adjunction $(\opb{\hat{f}},\oim{\hat f})$, we get the desired morphism
\eqn
f^{\temp, \, \sharp}:\opb{(f^\temp)}(\Cinfn{X_1^\temp}{\mathrm{t}\;(0,\bullet)},\; \ol\partial) \to (\Cinfn{X_0^\temp}{\mathrm{t}\;(0,\bullet)},\;\ol\partial)). 
\eneqn
Moreover, this morphism does not depend of the choice of a compactification.
\end{proof}

The datum of $(f^\temp, f^{\temp \, \sharp})$ defines a morphism in $\mathsf{TFS}$. If there is no risk of confusion, we write $f^\temp$ instead of $(f^\temp, f^{\temp \, \sharp})$. Finally, if $g \cl Y_0 \to Y_1$ is an other regular morphism, on checks that $(g \circ f)^\temp=g^\temp \circ f^\temp$ and that $\id^\temp=\id$.

In view of the preceding construction we can state the following Theorem.
\begin{theorem}
The functor $(\cdot)^\temp \cl\mathsf{Var_{sm}} \to \mathsf{TFS}$, $(X, \sho_X) \mapsto (X^\temp,\Otempx)$, $f \mapsto f^\temp$ is well defined and is called the functor of \textit{tempered analytification} ($\mathrm{t}$-analytification for short). We denote by $\mathsf{TAn}_\C$ its image in $\mathsf{TFS}$.
\end{theorem}

\begin{remark}
In a next paper, we will refine the construction of the functor $( \cdot )^\temp$ to get a functor valued in the $\infty$-category of $\mathbb{E}_\infty$ ringed spaces through which the usual analytification functor will factor. We prove such a result for the Stein tempered analytification functor in the next section.
\end{remark}

\subsection{The sheaves $\Otempx$ and $\Ozx$}\label{subsec:OtemOzar}
Let $(X,\sho_X)$ be a smooth algebraic variety. If we want to emphasize that $X$ is endowed with the Zariski topology we write $X_\zar$ instead of $X$.

We first note that there is a natural morphism of site 
\eq
\rho_\tz \colon X^\temp \to X_\zar.
\eneq
We shall study its properties.

Let $Y$ be a smooth algberaic compactification of $X$. There are  natural morphisms of sites that fit into the following commutative diagram.
\eqn
\xymatrix{
  Y_\sa \ar[r]^-{\rho_\zar^\sa} \ar[d]_-{j_{X_{Y_\sa}}} &Y_\zar \ar[d]^{j_{X_\zar}}.\\
  X^\temp \ar[r]_-{\rho_\tz} & X_\zar&
}
\eneqn
\begin{lemma}
Let $X$ be a smooth algebraic variety and $U$ a Zariski open subset of $X$. Let $f \in \sho_X(U)$. Then, the analytification of $f$ is a tempered holomorphic function.
\end{lemma}

\begin{proof}
Let $p \in X$ and $V$ an open affine subset of $X$ containing $p$, we can further assume that $U$ is affine and since $X$ is an algebraic variety, the open set $U \cap V$ is again affine.
In view of Proposition \ref{lem:precomposition}, we can assume that $V$ is a smooth affine subvariety of $\mathbb{A}^n_\C$. Then $f$ is the restriction to $U\cap V$ of a rational function $P(x)/Q(x)$ on $\mathbb{A}^n_\C$. It follows immediately from Lojaciewicz's inequality that $P(x)/Q(x)$ is tempered in $p$ when $p$ is a zero of $Q$.
\end{proof}

Using the commutative differential graded algebra \eqref{eq:dolbeault2} to represent $\Oty$, the above observations implies that there is a canonical morphism of sheaves of differential graded rings
\eqn\label{mor:Liouvilleinv}
 \Ozy \to \oim{\rhoz^{\sa}}\Oty\, , \quad \varphi \mapsto \varphi^\an.
\eneqn
Restrincting to $X_\zar$ and using the adjunction $(\opb{\rho_\tz}, \oim{\rho_\tz})$ this induces a morphism
\eq
L_\temp \cl \opb{\rho_\tz}\Ozx \to \Otempx.
\eneq
By adjunction, we get the morphism
\eq\label{mor:Liouvillesa}
 \Ozx \to \roim{{\rho_\tz}}\Otempx.
\eneq

Let $(X^\prime, \sho_{X^\prime})$ be a complex manifold and $Z^\prime$ be an anlytic subset of $X^\prime$. We denote by $\shi_{Z^\prime}$ the sheaf of holomorphic functions vanishing on $Z^\prime$ and set 
\eqn
\sect_{[Z^\prime]}(\sho_{X^\prime}) \eqdot \indlim[k] \hom[\sho_{X^\prime}](\sho_{X^\prime}/\shi^k_{Z^\prime},\sho_{X^\prime}).
\eneqn

\begin{lemma}
Let $X$ be a smooth algebraic variety and $Z$ be a closed subset of $X$. Then,
\eq
\rsect_{[Z^\an]}(\sho_{X^\an}) \simeq (\rsect_{Z}(\sho_{X_\zar}))^\an.
\eneq
\end{lemma}

\begin{proof}
Writing $\shi_{Z^\an}$ for the defining ideal of $Z^\an$ in $X^\an$, we have that
\eqn
\sect_{[Z^\an]}(\sho_{X^\an}) \eqdot \indlim[k] \hom[\sho_{X^\an}](\sho_{X^\an}/\shi^k_{Z^\an},\sho_{X^\an}) \simeq (\indlim[k] \hom[\sho_{X}](\sho_{X}/\shi^k_{Z},\sho_{X}))^\an.
\eneqn
Moreover by \cite[Theorem 2.8]{GrL}, on an algebraic variety,
\eqn
\rmH^i_{Z}(\sho_{X}) \simeq \indlim[k] \ext[\sho_{X}]{i}(\sho_{X}/\shi^k_{Z},\sho_{X}).
\eneqn
\end{proof}
\begin{theorem}\label{thm:Liouvillesheafsa} The object $\roim{\rho_\tz}\Otempx$ is concentrated in degree zero and
the morphism \eqref{mor:Liouvillesa} is an isomorphism. 
\end{theorem}
\begin{proof}
It is enough to check the isomorphism
\eqn
&&\rsect(U;\Oty)\simeq\rsect(U;\Ozy)
\eneqn
for any $U\in\Op_\Xz$. Let $Z\eqdot Y\setminus U$. There is the following commutative diagram 
\eqn
\xymatrix{
\rsect_Z(Y;\Ozy)  \ar[r] \ar[d]^{L_\temp} & \rsect(Y;\Ozy) \ar[r] \ar[d]^{L_\temp}& \rsect(U;\Ozy)\ar[d]^{L_\temp} \ar[r]^-{[+1]}&\\
\rsect_Z(Y;\Oty) \ar[r]& \rsect(Y;\Oty) \ar[r]& \rsect(U;\Oty)\ar[r]^-{[+1]}&.
}
\eneqn
where the rows are distinguished triangle.
As $Y$ is proper $\rsect(Y;\Oty)\simeq\rsect(Y^\an;\sho_{Y^\an})$. Then, $\rsect(Y;\Oty)\simeq\rsect(Y;\Ozy)$ by the GAGA theorem~\cite[Expos\'e XII]{SGA1}. We are reduced to prove the local isomorphism
\eqn
&&\opb{\rhosa}\rsect_Z(\Oty)\simeq\rsect_{[Z^\an]}(\sho_{Y^\an}).
\eneqn
This result is Theorem 5.12 in~\cite{KS96}. This implies in particular that the morphism $\Ozy \to \roim{\rhoz^\sa}\Oty$ is an isomorphism.
\end{proof}

\begin{corollary}
A tempered holomorphic function on a Zariski open set is regular.
\end{corollary}

\begin{corollary}
Let $X$ and $Y$ be two smooth algebraic varieties and $f \cl X^\an \to Y^\an$ an holomorphic map. Assume that $Y$ is affine and that $f$ induces a morphism of ringed $\TP$-spaces
\begin{align*}
f \cl X^\temp \to Y^\temp, \quad \Op_{Y^\temp} \ni U \mapsto \opb{f}(U) \in \Op_{X^\temp}\\
\rmH^0(f^{\temp\, \sharp}) \cl \opb{f}\rmH^0(\Otempy) \to \rmH^0(\Otempx), \quad \varphi \mapsto \varphi \circ f .
\end{align*}
Then $\varphi$ is the analytification of a regular morphism.
\end{corollary}

\begin{proof}
Since $Y$ is affine, there is an algebraic closed immersion $j_Y \cl Y \to \mathbb{A}_\C^n$. By composition, we get a tempered morphism $j_Y f \cl X \to \mathbb{A}_\C^n$ and applying this morphism to the coordinate functions $x_i \cl \mathbb{A}^n_\C \to \C$ with $1 \leq i \leq n$, we get that  $x_i( j_Y f )\in \Otempy(Y)=\Ozy(Y)$. Thus, $f$ is algebraic.
\end{proof}

\section{Stein tempered analytification}\label{sec:sta}

\subsection{The subanalytic Stein topology}

In view of Conjecture \ref{claim:dolbeautvan} it is  natural to introduce the $\sas$-topology defined below. In all of Section \ref{sec:sta}, we are not assuming Conjecture \ref{claim:dolbeautvan} unless explicitely stated.  

\begin{definition}
Let $X$ be a complex manifold. We denote by $\Xsas$ the presite for which the open sets are the finite unions of  $U\in\Op_\Xsa$ with $U$ Stein. We endow $\Xsas$ with the topology induced by $\Xsa$. 
\end{definition}
We have the natural morphisms of sites  
\eqn
\xymatrix{
X \ar@/_1pc/[rr]_-{\rhosasa} \ar[r]^-{\rhosa}&\Xsa \ar[r]^-{\rhosas} & \Xsas
}
\eneqn
where $\rhosasa\eqdot \rhosas \circ \rhosa$.

The morphism of sites $\rhosas$ induces the following pairs of adjoint functors
\eq\label{eq:fctrhosas}
\xymatrix{
  \opb{\rhosas} \cl \md[\cor_{X_\sas}]\ar@<.5ex>[r]& \md[\cor_\Xsa]\ar@<.5ex>[l] \cl \oim{\rhosas}
}
\eneq
and 
\eq
\xymatrix{
\opb{\rhosas} \cl \Derp(\cor_{X_\sas}) \ar@<.5ex>[r]& \Derp(\cor_\Xsa)\ar@<.5ex>[l] \cl \roim{\rhosas}.
}
\eneq
Note that the functors $\oim{\rhosas}$ and $\oim{\rhosa}$ commute with $\hom$ and preserve injectives. The last observation implies that $\roim{\rhosasa}\simeq \roim{\rhosas}\circ \roim{\rhosa}$.

\begin{proposition}
\banum
\item The functor $\oim{\rhosasa}$ and $\roim{\rhosasa}$ are fully faithful.
\item The functor $\oim{\rhosas}$ and $\roim{\rhosas}$ are fully faithful.
\eanum
\end{proposition}

\begin{proof}

\banum
\item This follows from the fact that any point $x \in X$ has a fundamental system of neighbourhood composed of Stein subanalytic open sets. For details see \cite[Proposition 6.6.1]{KS01}. This, together with the exactness of $\opb{\rhosasa}$ implies that $\opb{\rhosasa}\,\roim{\rhosasa} \to \id$ is an isomorphism. 

\item By definition of $\rhosasa$, $\oim{\rhosasa}= \oim{\rhosas} \oim{\rhosa}$. Since $\oim{\rhosasa}$ and $\oim{\rhosa}$ are fully faithful, the functor $\oim{\rhosas}$ is fully faithful. The proof for $\roim{\rhosas}$ is similar.
\eanum

\end{proof}

We will also need for technical matters a variant of the sas topology. These two topologies have equivalent categories of sheaves.
\begin{definition}
We denote  by $\Xsat$ the presite for which the open sets are the $U\in\Op_\Xsa$ with $U$ a Stein open subset. We endow $\Xsat$ with the topology induced by $\Xsa$.
\end{definition}

We have the natural morphisms of sites  
\eqn
\xymatrix{
\Xsas \ar[r]^-{\rhosat}& \Xsat
}
\eneqn

The functor $\oim{\rho_\sat}\cl \md[\cor_\Xsas] \to[\sim] \md[\cor_\Xsat]$ is an equivalence of categories the inverse of which $\rho^\ddagger_\sat\cl \md[\cor_\Xsat] \to[\sim] \md[\cor_\Xsas]$ is given by
\eqn 
&& \rho^\ddagger_\sat F(U) \eqdot \prolim[\stackrel{U_i \to U}{U_i \in \Op_\Xsat}]F(U_i).
\eneqn
More generally, the topos associated to $\Xsas$ and $\Xsat$ are equivalent.

We set
\eq \label{eq:notsas}
\Otsx\eqdot \rmH^0(\roim{\rhosas}\Otx)
\eneq

\begin{remark}\label{rem:concentration} Assuming the Conjecture \ref{claim:dolbeautvan}, we deduce that the object  $\roim{\rhosas}\Otx \in\Derb(\C_\Xsas)$ is concentrated in degree $0$.
To check it, it is enough to prove that any $U\in\Op_\Xsas$ admits a finite covering $U=\bigcup_iU_i$ with 
$\rsect(U_i;\roim{\rhosas}\Otx)$ concentrated in degree $0$. We may assume $U$ is Stein and we cover $\ol U$ with a finite union 
of open sets $V_i$, $i\in I$, such that $V_i\in\Op_\Xsa$, $V_i$ is Stein and $V_i$ is contained and relatively compact in a chart $\phi_i\cl W_i\isoto W'_i\subset\C_\sa^n$. Set $U_i=U\cap V_i$ and $U'_i=\phi_i(U_i)$. Then 
$\rsect(U_i;\roim{\rhosas}\Otx)\simeq \rsect(U'_i;\Ot[\C^n])$ and this last complex is concentrated in degree $0$. 

This justifies the notation~\eqref{eq:notsas}.
\end{remark}

Let $U\in\Op_\Xsas$. We endow $U$ with the topology induced by $\Xsas$ and we denote this site by $U_\Xsa$. There is a natural morphsim of sites $ U_\sas\to U_{X_\sas}$. In general, this morphism is not an equivalence of site.

\begin{proposition}\label{prop:morphcompactsas}
Consider a morphism of complex manifolds $f\cl X\to Y$ and let $U \in \Op_{\Xsas}$ and $V \in \Op_{\Ysas}$. Assume that $f$ induces an isomorphism of complex analytic manifolds $f_U \cl U \isoto V$. Then, $f_U$ induces 
\banum
\item  an isomorphism of sites ${f_U} \cl U_{X_{\sas}} \to V_{Y_{\sas}}$\,, given by the functor 
\eqn
f_U^t \cl \Op_{U_{X_\sas}} \to \Op_{V_{Y_{\sas}}}, \quad W \mapsto \opb{{f_U}}(W)=\opb{f}(W) \cap U,
\eneqn
\item an ismorphism of sheaves of rings
\eq\label{iso:otsas}
\oim{{f_U}}\Ot[X_{\sas}]\vert_{U_{X_{\sas}}} \simeq \Ot[{Y_{\sas}}]\vert_{V_{Y_{\sas}}}.
\eneq
\eanum
\end{proposition}

\begin{proof}
\banum

\item clear.

\item I follows from Proposition \ref{prop:morphcompact} and Remark \ref{rem:ringstructsa} that we have an isomorphism of rings
\eq\label{iso:prelimsas}
\oim{f\vert_U}\rmH^0(\Otx)\vert_\Uxsa &\simeq &\rmH^0(\Otx)\vert_\Vysa.  
\eneq

Applying the functor $\oim{\rho_{V_{Y_\sas}}}$ to the isomorphism \eqref{iso:prelimsas} and using the below commutative diagram of morphism of sites provides the isomorphism \eqref{iso:otsas}.

\eqn
\xymatrix {  & \Ysa \ar[rr]^-{\rhosas} \ar[dd]_(0.65){j_{V_\Xsa}}|!{[dl];[dr]}\hole &   & \Ysas \ar[dd]^{j_{V_\Ysas}} \\
           \Xsa \ar[ru]^{f} \ar[dd]_(0.35){j_{U_\Xsa}} \ar[rr]^(.65){\rhosas} &   & \Xsas \ar[ru] \ar[dd]^(0.35){j_{U_\Xsas}} &   \\
             & V_\Ysa \ar[rr]^(0.35){\rho_{V_\Ysas}}|!{[ur];[dr]}\hole &   & V_\Ysas \\
           U_\Xsa \ar[ru]^{f_U} \ar[rr]_{\rho_{U_\Xsas}}&   & U_\Xsas \ar[ru]^{f_U} & 
    }
\eneqn

\eanum
\end{proof}

\subsection{The Stein tempered analytification functor}
We keep the notation of Section \ref{section:tempered}. The aim of this section is to construct the Stein tempered analytification functor. For that purpose, we need the following lemma.

\begin{lemma}\label{lem:indepsas}
Let $X$ be a smooth algebraic variety and $Y$ a smooth algebraic compactification of $X$.
\banum
\item The site $X_{Y_\sas}$ does not depend on the choice of a smooth algebraic compactification $Y$ of $X$,

\item The sheaf of rings $\Ot[Y_{\sas}]\vert_{X_\Ysas} \in \Mod(\cor_{X_\Ysas})$ does not depend on the choice of an algebraic compactification $Y$ of $X$.
\eanum
\end{lemma}

\begin{proof}
\banum
\item is a consequence of Proposition \ref{prop:morphcompactsas} (a) and Diagram \eqref{diag:dominationcomp}.
\item is a consequence of Proposition  \ref{prop:morphcompactsas} (b) and Diagram \eqref{diag:dominationcomp}.
\eanum
\end{proof}

As a consequence of Lemma \ref{lem:indepsas} we can state

\begin{definition}
Let $X$ be a smooth algebraic variety. The \textit{ Stein tempered analytification} (st-analytification for short) of $X$ is the ringed $\mathcal{TP}$-space $(X^\stan, \Otax)$ where
\eqn
&& X^\stan \eqdot (X(\C), \, X_{Y_\sas})  \quad \textnormal{and} \quad \Otax \eqdot\Ot[Y_\sas]|_{X_{Y_\sas}}.
\eneqn
\end{definition}

There is a natural morphism of sites $\rho_\stan^\temp \cl X^\temp \to X^\stan$ and $\Otax \simeq \oim{\rho_\stan^\temp}\rmH^0(\Otempx)$.\\

We associate to a morphsim of smooth algebriac varieties $f \cl X_0 \to X_1$ its Stein tempered analytification $f^\stan$. We first establish the following lemma.

\begin{lemma}\label{lem:morwelldefsas}
Let $f \cl X_0 \to X_1$ be a morphism of smooth algebraic varieties. Let $U \in \Op_{X_1^\stan}$. Then $(f^\an)^{-1}(U) \in \Op_{X_0^\stan}$.
\end{lemma}

\begin{proof}
There exist smooth algebraic compactifications $Y_0$ and $Y_1$ of $X_0$ and $X_1$ such that $f$ extends to a regular morphism $\widehat{f} \cl Y_0 \to  Y_1$ such that Diagram \eqref{diag:compmap} commutes. Assume that $U \subset X_1$ is a subanalytic Stein open subset of $Y_1$.
By Lemma \ref{lem:morwelldef}, $f^{-1}(U)$ is subanalytic in $Y_0$.

The Zariski open subset $X_0$ of $Y_0$ has a finite covering $(V_i)_{1 \leq i \leq p}$ by Stein subanalytic open subsets of $Y_0$ (For instance, take the analytification of a finite open affine covering of $X_0$). Let $f_i$ be the restriction of $f$ to $V_i$. Then $f_i^{-1}(U)=\opb{f}(U) \cap V_i$ is a subanalytic Stein open subset of $Y_0$ and $f^{-1}(U)=\bigcup_{1 \leq i \leq n} f_i^{-1}(U)$. Thus $f^{-1}(U)$ is a finite union of subanalytic Stein open subsets of $Y_0$. Finally, if $U$ is a finite union of Stein open subsets $(U_i)_{1 \leq i \leq m}$ which are subanalytic in $Y_1$,  the above argument applied to each $U_i$ proves that $f^{-1}(U)$ is again a finite union of Stein open subsets of $Y_0$.
\end{proof}

Let $f\cl X_0 \to X_1$ be a morphism of smooth algebraic varieties. By Lemma \ref{lem:morwelldefsas}, the morphism of sites $f^\temp \cl X_0^\temp \to X_1^\temp$ induces a morphism 
\eqn
f^\stan \cl X_0^\stan \to X_1^\stan, \;\Op_{X_1^\stan} \ni U \mapsto (f^\an)^{-1}(U) \in \Op_{X_0^\stan}.
\eneqn

\begin{lemma}
Let $f \cl X_0 \to X_1$ be a morphism of smooth algebraic varieties. Then $f$ induces a morphism of sheaves
\eqn
f^{\stan\,\sharp} \cl \opb{(f^\stan)}\Ot[X_1^\stan] \To \Ot[X_0^\stan] , \quad \varphi \mapsto \varphi \circ f^\stan.
\eneqn
\end{lemma}

\begin{proof}
This follows directly from Lemma \ref{lem:morsheaftemp}.
\end{proof}

The datum of $(f^\an, f^{\stan \, \sharp})$ define a morphism of ringed $\mathcal{TP}$-space. If there is no risk of confusion, we write $f^\stan$ instead of $(f^\stan, f^{\stan \, \sharp})$. Finally, if $g \cl Y_0 \to Y_1$ is an other regular morphism, on checks that $(g \circ f)^\stan=g^\stan \circ f^\stan$ and that $\id^\stan=\id$.

The above constructions give rises to the Stein tempered analytification functor.

\begin{definition}
The functor $(\cdot)^\stan \cl\mathsf{Var_{sm}} \to \mathsf{TRgS}$, $(X, \sho_X) \mapsto (X^\stan,\Otax)$, $f \mapsto f^\stan$ is well defined and is called the functor of \textit{Stein tempered analytification} ($\mathrm{st}$-analytification for short). We denote by $\mathsf{STAn}_\C$ its image in $\mathsf{TRgS}$
\end{definition}

\begin{remark}
The interest of the st-analytification functor is that it associates to a smooth algebraic variety, a ringed $\TP$-space. The structure sheaf of this space has, under the assumption of Conjecture \ref{claim:dolbeautvan}, the same cohomology that the sheaf of regular functions of the algebraic variety under consideration. This is not the case  for the functor associating to a smooth algebraic variety the ringed $\TP$-space $(X^\temp, \rmH^0(\Otempx))$ even under the assumption of Conjecture \ref{claim:dolbeautvan}. Moreover, the st-analytification functor produces a sheaf of rings whereas an enhanced version of the t-analytification functor would produce a sheaf of $\mathbb{E}_\infty$-ring.
\end{remark}

\subsection{Comparison between the different analytification functors}

A careful examination of the construction of the functor $(\cdot)^\stan$ shows that the following functor is well defined.
\eqn
\Stn: \mathsf{TAn}_\C \to \mathsf{STAn}_\C, \quad (X^\temp, \Otempx) \mapsto (X^\stan, \Otax), \;f^\temp \mapsto f^\stan.
\eneqn
This means that $(\cdot)^\stan=\Stn \circ(\cdot)^\temp$.

We now study the relation between the usual  analytification functor, $\mathrm{t}$-analytification and $\mathrm{st}$-analytification functors.

Let $(X, \sho_X)$ be smooth algebraic variety. There is a natural morphism of sites
\eq
\rho_\stan \cl X^\an \to X^\stan.
\eneq
We associate to the st-analytification $(X^\stan,\Otax)$ of a smooth algebraic variety $(X, \sho_X)$, the ringed space $((X^\stan)^\mathrm{Top},\opb{\rho}_\stan\Otax)$. It is clear that the assignment

\eqn
\mathrm{u} \cl \mathsf{STAn}_\C \to \mathsf{An}_\C, \quad (X^\stan,\Otax)\mapsto ((X^\stan)^\mathrm{Top},\opb{\rho}_\stan\Otax)
\eneqn
is a well defined functor.

We remark that that the topological space $(X^\stan)^\mathrm{Top}$ is the underlying topological space of the complex manifold $(X^\an,\sho_{X^\an})$ since $X^\stan$ contains a basis of the topology of $X^\an$ and $X^\stan \subset X^\an$. Moreover, there is a morphism of sheaves
\eqn
\Otax \to \oim{\rho_\stan} \sho_{X^\an}
\eneqn
which induces by adjunction a morphism
\eq\label{mor:desan}
\opb{\rho_\stan}\Otax \to  \sho_{X^\an}.
\eneq
It follows from the fact that $\Ot[X^\stan_x]\simeq \sho_{X_x}$ that the morphism \eqref{mor:desan} is an isomorphism of sheaves of rings.
This implies the following proposition.

\begin{theorem}
The below diagram is quasi-commutative.
\begin{equation*}
\xymatrix{
\mathsf{TAn}_\C \ar[r]^-{\Stn}& \mathsf{STAn}_\C \ar[d]^-{\mathrm{u}} \\
\mathsf{Var_{sm}} \ar[u]^-{(\cdot)^\temp} \ar[ru]^-{(\cdot)^\stan} \ar[r]_-{(\cdot)^\an}& \mathsf{An}_\C
}
\end{equation*}
\end{theorem}

\subsection{The flatness of $\Otax$ over $\Ozx$}

Let $X$ be a smooth algebraic variety and $Y$ be a smooth algebraic compactification of $X$. There is a canonical morphism of sites
\eq
\rho_\stza \cl: X^\stan \to X_\zar.
\eneq
We study the properties of this morphism.

By the results of the subsection \ref{subsec:OtemOzar}, we have the morphism \eqref{mor:Liouvilleinv}. Using that $\roim{\rho_\tz}=\roim{\rho_\stza}\roim{{\rho_\stan^\temp}}$ and taking the zero degree cohomology, we obtain the morphism

\eq
L_\stan \cl \opb{\rhostz}\Ozx \to \Otax\, , \quad \varphi \mapsto \varphi^\an.
\eneq
It is clear that $L_\stan$ does not depend of the choice of a compactification $Y$ of $X$. By adjunction, we get the morphism
\eq\label{mor:Liouville} \Ozx \to \oim{\rhoz}\Otax\, , \quad \varphi \mapsto \varphi^\an.
\eneq

\begin{corollary}\label{cor:Liouvillesheaf}
Let $X$ be a smooth algebraic variety, the morphism \eqref{mor:Liouville} is an isomorphism.
\end{corollary}

\begin{proof}
This is a direct consequence of Theorem \ref{thm:Liouvillesheafsa}
\end{proof}

\begin{remark}\label{rem:Liouvillefort}
Under the assumption of Conjecture \ref{claim:dolbeautvan}, one obtain the following stronger version of Corollary \ref{cor:Liouvillesheaf}.\\
\textit{
The morphism
\eqn
 \Ozx \to \roim{\rho_\stza}\Otax.
\eneqn
is an isomorphism.
}
\end{remark}

\subsubsection{A flatness result}

In this subsection, we prove the flatness of $\Otax$ over $\opb{\rhostz}\Ozx$.

\begin{lemma}\label{le:flatanzar}
Let $X$ be a smooth affine algebraic variety and let $W$ be an open subset of $X_{\mathrm{zar}}$. Then $\Ozx(W)$ is Noetherian.
\end{lemma}

\begin{proof}
Let $Z$ be the complement of $W$ in $X$. The set $Z$ is an algebraic subset of $X$ and $Z=Z_{\geq 2} \cup Z_{=1}$ where $Z_{\geq 2}$ is the union of the irreducible components of $Z$ of codimension at least two and $Z_{=1}$ the union of the irreducible components of codimension exactly one. Let $\widetilde{W}=X \setminus Z_{=1}$. By construction $U \subset W \subset \widetilde{W}$ and since $\widetilde{W}$ is the complement of a closed subset purely of codimension one in a smooth affine variety $\widetilde{W}$, it is affine. Since $X$ is smooth it follows from Riemann second extension theorem that $\Ozx(W) \simeq \Ozx(\widetilde{W})$. This implies that $\Ozx(W)$ is Noetherian.
\end{proof}

We will need the following result.

\begin{lemma}[{\cite[Lemma 8.13]{Porta}}]\label{PY}
Let $X$ be a Stein manifold and $U$ be a relatively compact Stein open subset of $X$. Then $\sho_X(U)$ is flat over $\sho_X(X)$.
\end{lemma}

\begin{lemma}
Let $X_\zar$ be a smooth affine algebraic variety and $U \in \Op_{X^\stan}$ be relatively compact Stein open subset of $X^\an$. Then $\sho_{X^\an}(U)$ is flat over $\sho_{X_\zar}(X_{\zar})$.
\end{lemma}

\begin{proof}
It follows from Lemma \ref{PY} that $\sho_{X^\an}(U)$ is flat over $\sho_{X^\an}(X)$. Thus it is sufficient to prove that $\sho_{X^\an}(X)$ is flat over $\Ozx(X)$. Since $\Ozx(X)$ is Noetherian, we just need to show that for every exact sequence
\eq\label{eq:shortO}
&& \Ozx(X)^L\to[A]\Ozx(X)^M\to[B]\Ozx(X)^N
\eneq
the exact sequence obtained by applying $\sho_{X^\an}(X) \tens[\Ozx(X)] (\cdot)$ to the sequence \eqref{eq:shortO} is exact. As $X$ is an affine variety the sequence \eqref{eq:shortO} is equivalent to an exact sequence of sheaves
\eq\label{eq:short1}
&&\sho_{X_\mathrm{zar}}^L\to[A]\sho_{X_\mathrm{zar}}^M\to[B]\sho_{X_\mathrm{zar}}^N.
\eneq
This provides us with the following commutative diagram
\begin{center}
\scalebox{0.8}{\parbox{\linewidth}{
\eqn
\xymatrix{
\sho_{X^\an}(X) \tens[\sho_X(X)] \sho_{X}(X)^L\ar[r]^-A \ar[d]^-{\wr} & \sho_{X^\an}(X) \tens[\sho_X(X)] \sho_{X}(X)^M \ar[d]^-{\wr} \ar[r]^-B&\sho_{X^\an}(X) \tens[\sho_X(X)]\sho_{X}(X)^N \ar[d]^-{\wr}\\
\Gamma(X;\sho_{X^\an} \tens[\opb{\rhostz}\sho_X]\opb{\rhostz}\sho_{X}^L) \ \ar[r]^-A&\Gamma(X;\sho_{X^\an} \tens[\opb{\rhostz}\sho_X]\opb{\rhostz}\sho_{X}^M)\ar[r]^-B&\Gamma(X;\sho_{X^\an} \tens[\opb{\rhostz}\sho_X]\opb{\rhostz}\sho_{X}^N).
}
\eneqn
}}
\end{center}
The bottom line of the diagram is exact since $\sho_{X^\an}$ is flat over $\Ozx$ and $\Gamma(X; \cdot)$ is exact since $X$ is Stein.
\end{proof}

\begin{lemma}\label{le:partialflatness}
Let $X$ be a smooth affine algebraic variety, let $U \in Op_{X^\stan}$ be a relatively compact Stein open subset of $X^\an$ and $W$ be a Zariski open subset of $X$ containing $U$. Then $\Otax(U)$ is flat over $\sho_{X_\zar}(W)$.
\end{lemma}
\begin{proof}
Consider an exact sequence 
\eqn
&& \Ozx(W)^L\to[A]\Ozx(W)^M\to[B]\Ozx(W)^N
\eneqn
where $A$ and $B$ are two matrices with coefficients in $\Ozx(W)$. Tensoring by $\Otax(U)$, we obtain the exact sequence
\eq \label{diag:suitetemp}
&& \Otax(U)^L\to[A]\Otax(U)^M\to[B]\Otax(U)^N.
\eneq
Let $w \in \Otax(U)^M$ such that $Bw=0$. Then, by Lemma \ref{le:flatanzar}, there exists $v \in \sho_{X^\an}^L(U)$ such that $Av=w$. Since $X$ is affine  there exists $f_1, \ldots, f_n \in \Ozx(X)$ such that $W=D(f_1)\cup \ldots \cup D(f_n)$ where $D(f_i)$ is the distinguished open set associated to $f_i$. Thus, there exists $m_1, \ldots m_n \in \N$ such that $f_1^{m_1} \ldots f_n^{m_n}A$ is an $M \times L$ matrix with coefficients in $\Ozx(X)$. Moreover $f_1^{m_1} \ldots f_n^{m_n}w \in \Otax(U)^M$ since the $f_i^{m_i}$ are regular functions and $f_1^{m_1} \ldots f_n^{m_n}Av=f_1^{m_1} \ldots f_n^{m_n}w$. It follows from \cite[Theorem 2]{Si70} that there exists $u \in \Otax(U)^L$ such that $f_1^{m_1} \ldots f_n^{m_n}Au=f_1^{m_1} \ldots f_n^{m_n}w$. This implies that $Au=v$ on $U \cap D(f_1) \cap \ldots \cap D(f_n)=U \setminus \mathcal{Z}(f)$ with $f=f_1\ldots f_n$. Since $f$ is defined on $X$, $Au=v$ on $U$. It follows that the sequence \eqref{diag:suitetemp} is exact. As $\Ozx(W)$ is Noetherian, this implies that $\Otax(U)$ is flat over $\Ozx(W)$.
\end{proof}

For a sheaf $F$ on $\Xz$, that is, $F\in\md[\C_{\Xz}]$, denote by $\pshopb{\rhostz}F$ the inverse image of $F$ in the category of presheaves on $X^\stan$. Then for $U\in\Op_{X^\stan}$, 
\eqn
&&\pshopb{\rhostz}F(U)=\indlim[W]F(W)
\eneqn
where $W$ ranges over the family of objects of $\Op_{\Xz}$ such that $U\subset W$.\\

\begin{theorem}\label{th:otaxflat}
The sheaf of rings $\Otax$ is flat over $\opb{\rhostz}\Ozx$.
\end{theorem}
\begin{proof}

Let $Y$ be a smooth compactification of $X$. It sufficient to prove that $\Otsy$ is flat over $\opb{\rhostz}\Ozy$ and since $\Ysat$ is a dense subsite of $\Ysas$ this is equivalent to prove that $\oim{\rhosat}\Otsy$ is flat over $\oim{\rhosat}\opb{\rhostz}\Ozy$.

Let $(Y_i)_{i \in I}$ be an open affine covering of $Y$ and $(V_{ij})_{j \in J}$ be a covering of $Y_i$ by Stein open sets relatively compact in $Y_i$ and subanalytic in $Y$. The family $(V_{ij})_{i \in I, j \in J}$ is a Stein subanalytic open covering of $Y$ and since $Y$ is compact, we can extract from it a finite covering of $Y$, say $\{V_0,\ldots,V_\alpha\}$ from $(V_{ij})_{j \in J, i \in I}$. By construction, $\lbrace V_0,\ldots,V_\alpha \rbrace$ is a covering of $Y_\sat$.

Flatness being a local property, it sufficient to show that $\Otsy\vert_{V_{kY_\sas}}$ is flat over $\opb{\rhostz}\Ozy \vert_{V_{kY_\sas}}$, $0 \leq k \leq \alpha$.

To prove this, it is sufficient by \cite[Tag 03ET]{stackproject}, to show that the presheaf $\oim{\rhosat}\Otsy\vert_{V_{kY_\sas}}$ is flat over $\oim{\rhosat}\pshopb{\rhostz}\Ozy \vert_{V_{kY_\sas}}$ (Since $\Ysat$ is a dense subsite of $\Ysas$, sheafification commutes with $\oim{\rhosat}$, see \cite[Tag 03A0]{stackproject} for more details). This means, we have to show that for every $V^\prime \in\Op_\Ysat$ such that $V^\prime \subset V_k$, $\Otsy(V^\prime)$ is flat over $\pshopb{\rhostz}\Ozy(V^\prime)$. But
\eqn
&&\pshopb{\rhostz}\Ozy(V^\prime)=\indlim[W]\Ozy(W)
\eneqn
where $W$ ranges over the family of objects of $\Op_{Y_{\zar}}$ such that $V^\prime \subset W$. 

Thus by \cite[Tag 05UU]{stackproject}, it is sufficient to prove that for such a $W$ the ring $\Otsy(V^\prime)$ is flat over $\Ozy(W)$. We can further assume that there is an $i_0 \in I$ such that all the $W$ are contained in $Y_{i_0}$ and $V_k$ is relatively compact subset of $Y_{i_0}$. The result follows then by Lemma \ref{le:partialflatness}.

\end{proof}

\subsection{Towards a tempered GAGA theorem}

In this subsection, we present some results in the direction of a tempered version of the GAGA theorem of Serre. We assume all along this subsection that Conjecture \ref{claim:dolbeautvan} holds.

One defines the functor $\spb{\rhostz}\cl\md[\Ozx]\to\md[\Otax]$ by
\eq\label{eq:defspb}
 \spb{\rhostz}(\scbul) \eqdot \Otax \tens[\opb{\rhostz}\Ozx] \opb{\rhostz}( \scbul).
 \eneq
It follows from Theorem~\ref{th:otaxflat} that this functor is exact. 
We have the pairs of adjoint functors
\eqn
\xymatrix{
\md[\Otax]\ar@<.5ex>[r]^{\oim{\rhostz}}&\md[\Ozx]\ar@<.5ex>[l]^{\spb{\rhostz}}.
}
\eneqn

\textit{Assume Conjecture \ref{claim:dolbeautvan} and let  $X$ be a smooth algebraic variety over $\C$. The functor $\spb{\rhostz}$ is exact and fully faithful when restricted to $\mdc[\Ozx]$. Its essential image is contained in $\Mod_{\textnormal{lfp}}(\Otax)$ the full subcategory of $\md[\Otax]$ spanned by the $\Otax$-modules locally of finite presentation.}

 The only things that remains to prove in the above claim is that $\spb{\rhostz}$ is fully faithful when restricted to $\mdc[\Ozx]$. Under the assumption of Conjecture \ref{claim:dolbeautvan}, this follows immediately from remark \ref{rem:Liouvillefort} and from the fact that any coherent sheaf on a smooth variety admits locally a finite free resolution.\\

One shall notice that we do not ask $X$ to be proper. Hence, this statement may be considered as a kind of weak Serre GAGA theorem in the non proper case.\\ 
We would like to conclude this paper by two questions.

\medskip
\noindent \textbf{Question 1:} Is the sheaf $\Otax$ coherent?

\medskip
\noindent \textbf{Question 2:} Is the functor $\spb{\rhostz}$ essentially surjective? 

\medskip
We make the following straightforward observation concerning Question 2. If a $\Otax$-module is locally of finite presentation on a cover formed of Zariski open subsets then it is in the essential image of $\spb{\rhostz}$.

\appendix

\providecommand{\bysame}{\leavevmode\hbox to3em{\hrulefill}\thinspace}

\footnote{
\noindent Fran{\c c}ois~Petit, \textsc{Mathematics Research Unit, University of  Luxembourg} \par\nopagebreak \textit{email:} \texttt{francois.petit@uni.lu}
}

\end{document}